\documentclass[10pt]{amsart}
\usepackage{amssymb}
\usepackage{epsfig}
\usepackage{url}
\usepackage{setspace}
\usepackage{pdflscape}
\theoremstyle{plain}

\newtheorem{thm}{Theorem}[section]
\newtheorem{cor}[thm]{Corollary}

\newtheorem{prop}[thm]{Proposition}
\newtheorem{rem}[thm]{Remark}
\newtheorem{ques}[thm]{Question}
\newtheorem{conj}[thm]{Conjecture}

\def\cal{\mathcal}
\def\bbb{\mathbb}
\def\op{\operatorname}
\renewcommand{\phi}{\varphi}

\newcommand{\N}{\bbb{N}}

\newcommand{\Q}{\bbb{Q}}

\begin{document}

\title[On a Diophantine equation of Erd\H{o}s and Graham]{On a Diophantine equation of Erd\H{o}s and Graham}
\author{Szabolcs Tengely, Maciej Ulas and Jakub Zygad{\l}o}


\begin{abstract}
We study solvability of the Diophantine equation
\begin{equation*}
\frac{n}{2^{n}}=\sum_{i=1}^{k}\frac{a_{i}}{2^{a_{i}}},
\end{equation*}
in integers $n, k, a_{1},\ldots, a_{k}$ satisfying the conditions $k\geq 2$ and $a_{i}<a_{i+1}$ for $i=1,\ldots,k-1$. The above Diophantine equation (of polynomial-exponential type) was mentioned in the monograph of Erd\H{o}s and Graham, where several questions were stated. Some of these questions were already answered by Borwein and Loring. We extend their work and investigate other aspects of Erd\H{o}s and Graham equation. First of all, we obtain the upper bound for the value $a_{k}$ given in terms of $k$ only. This mean, that with fixed $k$ our equation has only finitely many solutions in $n, a_{1},\ldots, a_{k}$. Moreover, we construct an infinite set $\cal{K}$, such that for each $k\in\cal{K}$, the considered equation has at least five solutions. As an application of our findings we enumerate all solutions of the equation for $k\leq 8$. Moreover, by applying greedy algorithm, we extend Borwein and Loring calculations and check that for each $n\leq 10^4$ there is a value of $k$ such that the considered equation has a solution in integers $n+1=a_{1}<a_{2}<\ldots <a_{k}$. Based on our numerical calculations we formulate some further questions and conjectures.
\end{abstract}

\maketitle

\section{Introduction}\label{sec1}
In the very interesting book \cite{ErdGra} Erd\H{o}s and Graham stated many number theoretic problems. Some of them are related to Diophantine equations. At page 63 of this book the authors consider the following non-standard Diophantine equation
\begin{equation}\label{maineq}
\frac{n}{2^{n}}=\sum_{i=1}^{k}\frac{a_{i}}{2^{a_{i}}},\quad\mbox{where}\quad\;k>1,
\end{equation}
which can be seen as an equation of polynomial-exponential type. The authors stated some questions concerning this equation. For example, they asked whether for each $n\in\N_{+}$ there is a solution of (\ref{maineq}), i.e., we look for solutions in $k, a_{1},\ldots, a_{k}$; or whether for each $k\in\N_{+}$ there is a solution of (\ref{maineq}), i.e., we look for solutions in $n, a_{1},\ldots,a_{k}$. Moreover, the considered related problem of representations of real numbers $\alpha\in (0,2)$ in the form
$$
\alpha=\sum_{i=1}^{\infty}\frac{a_{i}}{2^{a_{i}}}.
$$
These questions were investigated by Borwein and Loring in \cite{BL}. In particular, in the cited paper, the authors proved that for each $k$ there is a solution of (\ref{maineq}). Moreover, they proposed an algorithm which, for a given $n$, allows to find (conjecturally finite) representation of $n/2^{n}$ in the form $\sum_{i=1}^{k}\frac{a_{i}}{2^{a_{i}}}$ \cite[Conjecture 1]{BL}. However, they do not investigate other Diophantine questions related to (\ref{maineq}). In particular, we are interested in the following questions, which in the light of findings of Borwein and Luring are quite natural.

\begin{ques}\label{ques1} What can be said about the number of solutions of {\rm (\ref{maineq})} when $k$ is fixed? Is it possible to enumerate all solutions of {\rm (\ref{maineq})} for small values of $k$?
\end{ques}

\begin{ques}\label{ques2}
Is it possible to bound $a_{k}$ in terms of $k$ only?
\end{ques}

Let us describe the content of the paper in some details. In Section \ref{sec2} we offer basic theoretical results concerning the solutions of (\ref{maineq}). We first enumerate all solutions of equation (\ref{maineq}) for $k\leq 8$. However, the most interesting part of this section is the proof of the inequality $a_k\leq 2^{k+2}+2k(\log_2 k-1)-4$. This answer Question \ref{ques2} affirmatively. In particular, for any given $k$, the considered equation has only finitely many effectively computable solutions in integers $n, a_{1},\ldots, a_{k}$.

In Section \ref{sec3}, by solving certain discrete logarithm problems we construct an infinite set $\cal{K}$, such that for each $k\in\cal{K}$ equation (\ref{maineq}) has at least five solutions in positive integers $n, a_{1},\ldots, a_{k}$. Moreover, we apply a modification of greedy strategy of Borwein and Loring and prove that for each $n\leq 10^4$ equation (\ref{maineq}) has a solution in positive integers $k, a_{1},\ldots, a_{k}$. As an application of our approach we construct an infinite set $\cal{R}$ of rational numbers, such that for each $x\in\cal{R}$ the number $x$ has at least nine representations in the form $\sum_{i=1}^{\infty}a_{i}/2^{a_{i}}$. Moreover, based on our numerical data we formulate precise conjecture concerning the quantity of $a_{k}$, i.e., $a_{k}\leq 2(n+k)$. We prove that our conjecture is true for all $n$ satisfying $n\geq 2^{k}-k$.

\section{Theoretical results}\label{sec2}

We start with some easy observations related to equation (\ref{maineq}).

\begin{thm}\label{thm1}
\begin{enumerate}
\item[(i)] Let $k$ be fixed. If the equation {\rm (\ref{maineq})} has a solution, then $n\leq 2^{k+1}-k-2$.

\item[(ii)] If {\rm (\ref{maineq})} holds then $n+1\leq a_{1}\leq n+3$ and $2^{a_{k}-a_{k-1}}|a_{k}$. Moreover, if $n\geq 2^{j+1}-j$ for some $1\leq j<k$, then
$$
a_i=n+i,\;for\;i=1,\ldots,j.
$$
\end{enumerate}
\end{thm}
\begin{proof}
Let us suppose that (\ref{maineq}) has a solution for $n$. It is clear that $a_{1}\geq n+1$ and thus $a_{i}\geq n+i$ for each $i\in\{1,\ldots,k\}$. Because the function $f(x)=x/2^{x}$ is decreasing for $x\geq 1$ We immediately get the inequality
$$
\frac{n}{2^{n}}=\sum_{i=1}^{k}\frac{a_{i}}{2^{a_{i}}}\leq \sum_{i=1}^{k}\frac{n+i}{2^{n+i}}=\frac{(2^{k}-1)n+2^{k+1}-k-2}{2^{n+k}}.
$$
By solving the resulting inequality we get the upper bound for $n$ in terms of $k$. Indeed, we have $n\leq 2^{k+1}-k-2$.

To prove the second part of our theorem let us suppose that $a_{1}\geq n+4$. Thus, $a_{i}\geq n+3+i$ for $i=1,\ldots, k$ and in consequence
$$
\frac{n}{2^{n}}=\sum_{i=1}^{k}\frac{a_{i}}{2^{a_{i}}}\leq \sum_{i=1}^{k}\frac{n+3+i}{2^{n+3+i}}=\frac{(2^{k}-1)n+5\cdot 2^k-k-5}{2^{n+k+3}}.
$$
By solving the resulting inequality with respect to $n$ we have
$$
n\leq \frac{5\cdot 2^k-6}{7\cdot 2^{k}+1}<1
$$
and get a contradiction. Thus $n+1\leq a_{1}\leq n+3$.

The divisibility $2^{a_{k}-a_{k-1}}|a_{k}$ is clear. Indeed, multiplying (\ref{maineq}) by $2^{a_{k-1}}$ we see that the number
$$
2^{a_{k-1}-n}n-\sum_{i=1}^{k-1}2^{a_{k-1}-a_{i}}a_{i}
$$
is an integer equal to $\frac{a_{k}}{2^{a_{k}-a_{k-1}}}$ and hence $2^{a_{k}-a_{k-1}}|a_{k}$.

Finally, to get the last statement we proceed by induction on $j\geq 1$. Let us start with $j=1$, so $n\geq 3$. If $a_1\neq n+1$, then $a_1\geq n+2$ and
$$
\frac{n}{2^n}=\sum_{i=1}^k\frac{a_i}{2^{a_i}}\leq\sum_{i=1}^k\frac{n+1+i}{2^{n+1+i}}<\sum_{i=n+2}^\infty\frac{i}{2^i}=\frac{n+3}{2^{n+1}},
$$ and so $n<3$, a contradiction.
Let us now take $n\geq 2^{j+2}-j-1$. Since also $n\geq 2^{j+1}-j$, by the induction hypothesis we get that $a_i=n+i$ for $1\leq i\leq j$. If $a_{j+1}\geq n+j+2$, then
\begin{align*}
\frac{n}{2^n}&=\sum_{i=1}^k\frac{a_i}{2^{a_i}}\leq\sum_{i=1}^j\frac{n+i}{2^{n+i}}+\sum_{i=j+1}^k\frac{n+1+i}{2^{n+1+i}}\\
             &<\frac{n(2^j-1)+2^{j+1}-j-2}{2^{n+j}}+\sum_{i=n+j+2}^\infty\frac{i}{2^i}\\
             &=\frac{n(2^j-1)+2^{j+1}-j-2}{2^{n+j}}+\frac{n+j+3}{2^{n+j+1}}=\frac{n(2^{j+1}-1)+2^{j+2}-j-1}{2^{n+j+1}}.
\end{align*}
 As a consequence we get $n<2^{j+2}-j-1$, a contradiction that completes the induction step.
\end{proof}

\begin{rem}
{\rm We observed that the necessary condition for solvability of (\ref{maineq}) is the inequality $n\leq 2^{k+1}-k-2$. This condition can not be improved. Indeed, as was observed by Borwein and Luring, if $k$ is fixed and $n=2^{k+1}-k-2$, then we have the equality
\begin{equation*}
 \sum_{i=1}^{k}\frac{n+i}{2^{n+i}}=\frac{n}{2^{n}},
\end{equation*}
i.e., equation (\ref{maineq}) has a solution $a_{i}=n+i, i=1,\ldots,k$.
}
\end{rem}
The divisibility property noted in the last part of Theorem \ref{thm1} can be strengthened as follows.

\begin{thm}
Let $a_1<\ldots<a_k$ be a solution to equation {\rm (\ref{maineq})} and $1\leq i\leq k-2$. Then
$$2^{a_k-a_i}\leq (a_{i+2}\cdot a_{i+3}\cdot\ldots\cdot a_{k-1}\cdot a_k)\cdot a_k$$
\begin{proof}
We first note that for $i=k-1$ we get $2^{a_k-a_{k-1}}\leq a_k$ by Theorem \ref{thm1}.(\ref{tw213}). We will proceed by induction on $j=k-i$, starting with $j=2$, i.e. $i=k-2$. Multiplying equation (\ref{maineq}) by $2^{a_{k-2}}$ we get
$$n\cdot 2^{a_{k-2}-n}=\sum_{s=1}^{k-2}a_s\cdot 2^{a_{k-2}-a_s}+\frac{a_{k-1}\cdot 2^{a_k-a_{k-1}}+a_k}{2^{a_k-a_{k-2}}}$$
and so $2^{a_k-a_{k-2}} | (a_{k-1}\cdot 2^{a_k-a_{k-1}}+a_k)$ and the second term is non-zero. Consequently $2^{a_k-a_{k-2}}\leq a_{k-1}\cdot 2^{a_k-a_{k-1}}+a_k\leq a_{k-1}\cdot a_k+a_k\leq (a_k-1)\cdot a_k+a_k=a_k^2$.

In the induction step we perform the same calculations. The only difference is that we multiply equation (\ref{maineq}) by $2^{a_{k-j}}$. As a result we obtain the equality
$$
n\cdot 2^{a_{k-j}-n}=\sum_{s=1}^{k-j}a_s\cdot 2^{a_{k-j}-a_s}+\frac{a_{k-j+1}\cdot 2^{a_k-a_{k-j+1}}+\ldots+a_{k-1}\cdot 2^{a_k-a_{k-1}}+a_k}{2^{a_k-a_{k-j}}},
$$
and thus $2^{a_k-a_{k-j}}|a_{k-j+1}\cdot 2^{a_k-a_{k-j+1}}+\ldots+a_{k-1}\cdot 2^{a_k-a_{k-1}}+a_k$. Now by the induction hypothesis:
\begin{align*}
&2^{a_k-a_{k-j}}\leq a_{k-j+1}\cdot a_{k-j+3}\cdot\ldots\cdot a_{k-1}a_k^2+\ldots+a_{k-2}a_k^2+a_{k-1}a_k+a_k\leq\\
&\leq (a_{k-j+2}-1)\cdot a_{k-j+3}\cdot\ldots\cdot a_{k-1}a_k^2+\ldots+(a_{k-1}-1)a_k^2+(a_k-1)\cdot a_k+a_k\leq\\
&\leq a_{k-j+2}\cdot a_{k-j+3}\cdot\ldots\cdot a_{k-1}a_k^2
\end{align*}
and the result follows.
\end{proof}
\end{thm}

As an immediate consequence from the above result we get the following.

\begin{cor}\label{corbound}
Let $a_1<\ldots<a_k$ be a solution to equation {\rm (\ref{maineq})}. We have $\frac{a_k^{k-1}}{2^{a_k}}\geq 2^{-a_1}$.
\end{cor}

Using the last part of Theorem \ref{thm1} and Corollary \ref{corbound} we can try to list all possible solutions of equation (\ref{maineq}) for small values of $k$.
Indeed, part (i) of Theorem \ref{thm1} provides a bound on values of $n$ to check, while part (ii) gives exact values for starting elements in $(a_{i})$ for large $n$. However, for small values of $n$, part (ii) of Theorem \ref{thm1} is not effective and we need to apply Corollary \ref{corbound} to further reduce the set of possible solutions by bounding $a_k$ depending on $a_1$.
After some considerable amount of computer calculations (case $k=8$ took more than two days of computing time) we provide complete list of solutions of equation {\rm (\ref{maineq})} for $k\leq 8$:

\begin{thm}\label{smallk}
 Let $k\in\{2,3,4,5,6,7,8\}$ and let us put $A=(a_{1},a_{2},\ldots, a_{k})$. All solutions of the equation {\rm (\ref{maineq})} are the following:
\begin{equation*}
\begin{array}{lll}
k=2: & [n,A]=   &[4,(5,6)]; \\
k=3: & [n,A]\in &\{[1,(3,6,8)], [1,(4,5,6)], [2,(3,6,8)], [2,(4,5,6)], [3,(4,6,8)], \\
     &          & \quad [11,(12,13,14)]\}; \\
k=4: & [n,A]\in & \{[9,(10,11,13,14)], [26,(27,28,29,30)]\}; \\
k=5: & [n,A]\in & \{[5,(6,7,11,13,14)], [6,(7,8,11,13,14)], [15,(16,17,18,21,22)], \\
     &          & \quad [57,(58,59,60,61,62)]\};\\
k=6: & [n,A]\in & \{[4,(5,7,8,11,13,14)], [12,(13,14,15,20,21,24)], \\
     &          & \quad [13,(14,15,16,20,21,24)], [21,(22,23,24,26,27,32)],\\
     &          & \quad [120,(121,122,123,124,125,126)]\};\\
k=7: & [n,A]\in & \{[1,(4,5,7,8,11,13,14)], [2,(4,5,7,8,11,13,14)],\\
     &          & \quad [7,(8,9,11,15,20,21,24)], [18,(19,20,21,23,26,27,32)],\\
     &          & \quad [247,(248,249,250,251,252,253,254)]\};\\
k=8: & [n,A]\in & \{[17,(18,19,20,22,26,29,30,32)],\\
     &          & \quad [19,(20,21,22,24,26,29,30,32)],\\
     &          & \quad [197,(198,199,200,201,202,203,205,206)],\\
     &          & \quad [502,(503,504,505,506,507,508,509,510)]\}.
 \end{array}
\end{equation*}

\end{thm}

As a consequence of our computations we obtain an explicit family of rational numbers having at least three different representations in the form $\sum_{i=1}^{\infty}\frac{a_{i}}{2^{a_{i}}}$.

\begin{cor}\label{manyrep}
There are infinitely many values of $x\in\Q$ such that the number $x$ has at least three representations in the form
$$
x=\sum_{i=1}^{\infty}\frac{a_{i}}{2^{a_{i}}}.
$$
\end{cor}
\begin{proof}
From Theorem \ref{smallk} we see that for $n=1$ the equation (\ref{maineq}) has two solutions for $k=3$ and one solution for $k=7$. Let $(b_{i})_{i\in\N_{+}}$ be a sequence of positive integers satisfying $15\leq b_{1}$ and $b_{i}<b_{i+1}$ for $i=1, 2, \ldots$, and suppose that the number $x'=\sum_{i=1}^{\infty}\frac{b_{i}}{2^{b_{i}}}$ is rational. Then we have the representations
\begin{align*}
\frac{1}{2}+x'&=\frac{3}{2^{3}}+\frac{6}{2^{6}}+\frac{8}{2^{8}}+\sum_{i=1}^{\infty}\frac{b_{i}}{2^{b_{i}}}\\
              &=\frac{4}{2^{4}}+\frac{5}{2^{5}}+\frac{6}{2^{6}}+\sum_{i=1}^{\infty}\frac{b_{i}}{2^{b_{i}}}\\
              &=\frac{4}{2^4} + \frac{5}{2^5} + \frac{7}{2^7} + \frac{8}{2^8} + \frac{11}{2^{11}} + \frac{13}{2^{13}} + \frac{14}{2^{14}}+\sum_{i=1}^{\infty}\frac{b_{i}}{2^{b_{i}}},
\end{align*}
and from the assumption on the sequence $(b_{i})_{i\in\N_{+}}$ we know that the presented representations are different. In order to make the value of the sum $\sum_{i=1}^{\infty}\frac{b_{i}}{2^{b_{i}}}$ rational it is enough to take $b_{i}=pi+q$, where $p, q\in\N_{+}$ and $b_{1}=p+q>16$. Then we have
$$
\sum_{i=1}^{\infty}\frac{pi+q}{2^{pi+q}}=\frac{(q+p)2^p-q}{2^{q}(2^p-1)^2},
$$
a rational number.
\end{proof}

\begin{rem}
{\rm The above result was also obtained by Borwein and Luring using three representations of the number $\frac{1}{4}$. However, based on our computational approach, we will show in Corollary \ref{manyrepb} below, that in fact there are infinitely many rational numbers with at least nine representations of the form $\sum_{i=1}^{\infty}\frac{a_{i}}{2^{a_{i}}}$.}
\end{rem}

We close this section with the positive answer to Question \ref{ques2}.

\begin{thm}\label{thmbound}
Let $a_1<\ldots<a_k$ be a solution to equation {\rm (\ref{maineq})}. Then we have
$$a_k\leq 2n+2k\log_2 k.$$
\begin{proof}
It can be verified that the inequality holds for all solutions given in Theorem \ref{smallk}. So let us fix $k\geq 8$ and suppose that $a_k>2n+2k\log_2 k$ holds for some $n\geq 1$. Since $a_k>\frac{k-1}{\ln 2}$ and the function $f(x)=\frac{x^{k-1}}{2^x}$ is decreasing for $x>\frac{k-1}{\ln 2}$ we get that $\frac{a_k^{k-1}}{2^{a_k}}<\frac{(2n+2k\log_2 k)^{k-1}}{2^{2n+2k\log_2 k}}$. By Corollary \ref{corbound} this implies that
$$(2n+2k\log_2 k)^{k-1}>2^{2n+2k\log_2 k-a_1}\geq 2^{n-3+2k\log_2 k}$$
since $a_1\leq n+3$. But then
$$\Big(\frac{2n+2k\log_2k}{k^2}\Big)^{k-1}=\Big(\frac{2n+2k\log_2 k}{2^{2\log_2 k}}\Big)^{k-1}>2^{n-3+2\log_2 k}$$
We will show that this inequality is not satisfied for any $n$. Indeed, for $n=1$ we get:
$\big(\frac{2+2k\log_2k}{k^2}\big)^{k-1}\leq \big(\frac{2+16\log_2 8}{64}\big)^{k-1}<1$ and $2^{1-3+2\log_2 k}\geq 2^{-2+2\log_2 8}>1$, since $k\geq 8$ and considered functions are monotonic. Increasing $n$ by one, right hand side of the inequality is multiplied by 2, while left hand side by:
$$\Big(\frac{2n+2+2k\log_2 k}{2n+2k\log_2 k}\Big)^{k-1}=\Big(1+\frac{1}{n+k\log_2 k}\Big)^{k-1}\leq\Big(1+\frac{1}{2k}\Big)^{k-1}<e^{1/2}<2$$
This is a contradiction, so the inequality $a_k>2n+2k\log_2 k$ cannot hold.
\end{proof}
\end{thm}

\begin{cor}
Let $a_1<\ldots<a_k$ be a solution to equation {\rm (\ref{maineq})}. Then
$$a_k\leq 2^{k+2}+2k(\log_2 k-1)-4.$$
\end{cor}

\section{A computational approach to equation (\ref{maineq})  }\label{sec3}

We know that for any given $k$ the number of solutions of (\ref{maineq}), say $N(k)$, is bounded, and Theorem \ref{smallk} shows that
$$
N(2)=1,\;N(3)=6,\;N(4)=2,\;N(5)=4,\;N(6)=5,\;N(7)=3.
$$

We prove that there are infinitely many values of $k$ such that $N(k)\geq 5$. More precisely, we have

\begin{prop}\label{Nk}
If
\begin{eqnarray*}
k&\equiv & 10131316054712759135960334995313053617046\\
&&\pmod{20263657997642451746458664712008831939580},
\end{eqnarray*}
then the Diophantine equation {\rm (\ref{maineq})} has at least five solutions, i.e., $N(k)\geq 5.$
\end{prop}

Before we prove the above proposition, we describe the experimental strategy which we used. More precisely, we looked for values of $u$ and corresponding value(s) of $k$ such that there is a positive integer solution $n$ of the equation
$$
\frac{n}{2^{n}}=\sum_{i=1}^{k-2}\frac{n+i}{2^{n+i}}+\frac{n+k+u}{2^{n+k+u}}+\frac{n+k+u+1}{2^{n+k+u+1}},
$$
i.e., we look for integral values of the expression
$$
n=\frac{(2^{k-1}-k)(2^{u+3}-3)+3\cdot2^{k-1}+3u+1}{2^{u+3}-3}=2^{k-1}-k+\frac{3\cdot2^{k-1}+3u+1}{2^{u+3}-3}.
$$
Equivalently, we need to consider the polynomial-exponential congruence
\begin{equation}\label{cong}
3\cdot2^{k-1}+3u+1\equiv 0\pmod{2^{u+3}-3}.
\end{equation}

Note that if for a given $u$ the congruence (\ref{cong}) has a solution in $k$, then necessarily $k<r:=\op{ord}_{2^{u+3}-3}(2)$, where as usual $\op{ord}_{m}(a)=\op{min}\{v\in\N_{+}:\;a^{v}\equiv 1\pmod{m}\}$. In particular, if $k_{0}$ is a solution for $u$, then for each $t\in\N$ the number $k=rt+k_{0}$ is also a solution.
Congruence \eqref{cong} can be written as
$$
2^{k-1}\equiv\frac{-3u-1}{3}\pmod{2^{u+3}-3},
$$
hence one has to resolve a discrete logarithm problem.
There are exactly 16 values of $u\leq 120$ such that (\ref{cong}) has a solution, see table below.
{\tiny
\begin{center}
\begin{tabular}{|r|r|r|} \hline
$u$ & $k_0$ & $r$ \\ \hline
$0$ & $4$  & $4$ \\ \hline
$1$ & $5$ & $12$ \\ \hline
$2$ & $22$ & $28$ \\ \hline
$3$ & $48$ & $60$ \\ \hline
$4$ & $83$ & $100$ \\ \hline
$6$ & $221$ & $508$ \\ \hline
$9$ & $242$ & $4092$ \\ \hline
$11$ & $5531$ & $16380$ \\ \hline
$17$ & $66328$ & $1048572$ \\ \hline
$21$ & $2796185$ & $5592404$ \\ \hline
$22$ & $775376$ & $1116130$ \\ \hline
$26$ & $96489490$ & $536870908$ \\ \hline
$55$ & $5843993308712118$ & $26202761468337430$ \\ \hline
$99$ & $364550281031913286431277811782$ & $2535300206192230667655098198606$ \\ \hline
$113$ & $2452672773763126728478631379525174$ & $83076749736557242056487941267521532$ \\ \hline
$119$ & $3303995011423016739508338720636484139$ & $5316911983139663491615228241121378300$ \\ \hline
\end{tabular}
\begin{center}\small
Table 1. Solutions for $k$ of (\ref{cong}) for $u\leq 120$ together with the value of $r=\op{ord}_{2^{u+3}-3}(2)$.
\end{center}
\end{center}
}

\begin{proof}[Proof of {\rm Proposition \ref{Nk}}]
The idea of the proof is the following. If we write $f_{i}(x)=r_{i}x+k_{i}$, where $k_{i}, r_{i}$ correspond to $i$th elements in the table above, then to get values of $k$ such that (\ref{maineq}) has at least $m$ solutions, it is enough to find solutions of the system
$$
f_{i_{1}}(x_{1})=f_{i_{2}}(x_{2})=\ldots=f_{i_{m}}(x_{m})
$$
for certain $1\leq i_{1}<i_{2}<\ldots <i_{m}\leq 16$. We checked all 4368 combinations of five elements subsets of the set of linear functions $\{f_{1},\ldots,f_{15}\}$ and found that in each case, the above system has no solutions. In case of four functions we checked 1820 subsets and found exactly six subsets such that the above system has solutions. The simplest solutions are obtained in the case of the linear Diophantine system
\begin{eqnarray*}
&&28x_1 + 22=4092x_2 + 242=\\
&&=26202761468337430x_3 + 5843993308712118=\\
&&=2535300206192230667655098198606x_4 + 364550281031913286431277811782.
\end{eqnarray*}
The corresponding linear functions give (all) solutions of (\ref{cong}) for $u=2, 9, 55, 99$, respectively.
A standard method gives the solution
\begin{eqnarray*}
x_1&=&5192346432901574483898091387790622531230191866907+\\
&&+ 16989949871052950679749955447565756090108796474835t\\
x_2&=&35529252229043031659126725038645510966384499578+\\
&&+ 116255766468593015403958639426158643822836339515t\\
x_3&=&5548487715576217653155322603550910+\\
&&+ 18155284776542563811078158200217566t\\
x_4&=&57344569990628045006 +\\
&&+187637974874764336230t.
\end{eqnarray*}
where $t\in\N$, and the corresponding common value of $k$ is given by
\begin{eqnarray*}
k=k(t)&=& 10131316054712759135960334995313053617046+\\
&&+{20263657997642451746458664712008831939580t},
\end{eqnarray*}
In consequence, for given $u\in\{2,9,55,99\}$ and each $t\in\N$ we get an integer value of $n$ for $k=k(t)$ together with values of $a_{1},\ldots,a_{k}$ given by $a_{i}=n+i, i=1,\ldots, k-2, a_{k-1}=n+k+u, a_{k}=n+k+u+1$. Thus for any given $k=k(t)$ we have four solutions of (\ref{maineq}). One additional solution for $k$ corresponds to $n=2^{k+1}-k-2$ and $a_{i}=n+i, i=1,\ldots,k$.

\end{proof}

We finish our discussion with the following:

\begin{conj}\label{conjA}
Let us put
$$
\cal{U}=\{u\in\N_{+}:\;\mbox{congruence}\;(\ref{cong})\; \mbox{has a solution}\}.
$$
The set $\cal{U}$ is infinite.
\end{conj}

\begin{conj}\label{conjB}
We have $\limsup\limits_{k\rightarrow +\infty}N(k)=\infty$.
\end{conj}

Our proof of Proposition \ref{Nk} based on the existence of certain elements in the set $\cal{U}$. Thus, one can ask the following

\begin{ques}
Suppose that {\rm Conjecture \ref{conjA}} is true. Does {\rm Conjecture \ref{conjA}} implies {\rm Conjecture \ref{conjB}}?
\end{ques}

It seems that the most interesting (and difficult) question concerning equation (\ref{maineq}) is to whether, for a given $n$, there is $k\in\N_{+}$ such that (\ref{maineq}) has a solution. Essentially, this is \cite[Conjecture 1]{BL}. Unfortunately, we were unable to answer this question in full generality. Borwein and Luring proved that for each $n\leq 10^3$ equation (\ref{maineq}) has at last one solution.  We were able to extend the range of computations and prove the following:

\begin{thm}\label{10000}
For each $2\leq n\leq 10^4$ the Diophantine equation {\rm (\ref{maineq})} has a solution in variables $k, a_{1},\ldots, a_{k}$ satisfying $a_{i}=n+1$ and $a_{i}\geq n+i$ for $i=2,\ldots k$.
\end{thm}

We now describe a computational method which was used to get the above result. More precisely, in order to confirm that equation {\rm (\ref{maineq})} has a solution, the following "greedy" strategy was applied: assuming that we have found the sequence $(a_1,\ldots,a_l)$ such that $\frac{n}{2^n}>\sum_{i=1}^l\frac{a_l}{2^l}$, we define $a_{l+1}=j$ where $\frac{j}{2^j}$ is the first term that "fits", that is we take the smallest $j$ such that
$$
\frac{n}{2^n}\geq\sum_{i=1}^l\frac{a_i}{2^{a_i}}+\frac{j}{2^j},
$$
and hope that this process ends after a finite number of steps. Naive implementation of the procedure above leads to a very slow algorithm for large $n$, so we apply a different approach that is a slight modification of Algorithm 2 given in \cite{BL}.

Let $x\in\Q$, $0<x<2$ and define: $k_0=\min\{k\geq 1\colon \frac{k}{2^k}<x\}$.
We define a sequence $S(x)=(x_{k_0},x_{{k_0}+1},\ldots)$ as
$$x_{k_0}=x\cdot 2^{k_0-1}$$
and for $i\geq k_0$
$$
x_{i+1}=\begin{cases}2\cdot x_i-i,&\textnormal{ if }2\cdot x_i-i\geq 0,\\2\cdot x_i,&\textnormal{ otherwise.}
\end{cases}
$$
We say that the sequence $S(x)$ terminates if $x_i=0$ for some $i\geq k_0$ (and for all subsequent values in the sequence). It is not difficult to see that $x_i$ is precisely the numerator of the fraction
$$
\frac{x_i}{2^{i-1}}=x-\sum_{j=1}^{i-1}s_j\cdot\frac{j}{2^j},
$$
where $s_j=1$ if $\frac{j}{2^j}$ appears in the sum when applying the greedy strategy for $x$ (with the exception of $s_j=0$ if $x=\frac{j}{2^j}$) and $s_j=0$ otherwise. Moreover, if the sequence $S(x)$ terminates, then $s_j=1$ (i.e. $\frac{j}{2^j}$ appears in the representation of $x$) if and only if $x_{j+1}\neq 2x_j$.

The above algorithm was implemented in {\sc Mathematica} \cite{math} in the following form:

{\tt
\begin{verbatim}
greedy[x_, maxK_] :=   Module[{ind = {}, k = 0, v = x, n = 1},
    While[2*v < n + 1, v = 2*v; n++];
    While[v > 0 && k <= maxK, If[2*v - n >= 0,
    AppendTo[ind, n]; k++; v = 2*v - n, v = 2*v]; n++];
    If[v == 0, Return[ind]]]
\end{verbatim}
}

The value {\tt maxK} is the maximal value of $k$ which is used in calculations. Thus, if we evaluate {\tt greedy[$41/2^{41}$,10]}, then our program will terminate without any result. However, if we evaluate {\tt greedy[$41/2^{41}$,20]}, then our program returns
\begin{equation*}
\{{\tt 42, 43, 44, 45, 47, 49, 54, 55, 56, 61, 66, 68, 69, 70}\}.
\end{equation*}

Our observations show that to find a representation of $x$ using the greedy strategy we can calculate $S(x)$ and see if it terminates. This has the advantage of being much faster as the only operations involved are multiplication by 2 and subtraction (of integers if $x$ has power of 2 as the denominator). Moreover it can be easily verified that $x_i<i+1$ for all $i\geq k_0$ and so $x_i<k$ if the sequence $S(x)$ terminates after $k$ steps (or equivalently the representation for $x$ has $k$ terms), i.e. the numbers $x_i$ are feasible.

First of all, we note that our approach is strong enough to present an improvement of Corollary \ref{manyrep}. More precisely, we prove that
\begin{cor}\label{manyrepb}
There are infinitely many values of $x\in\Q$ such that $x$ has at least nine representations in the form
$$
x=\sum_{i=1}^{\infty}\frac{a_{i}}{2^{a_{i}}}.
$$
\end{cor}
\begin{proof}
To get the result it is enough to find one rational number $x$ with nine representations. The idea is very simple. Suppose that we have $x=\sum_{i=1}^{k_{1}}\frac{a_{1,i}}{2^{a_{1,i}}}$ and for $m\geq 2 $ we are able to compute the expansion
$$
\frac{a_{m-1,k_{m-1}}}{2^{a_{m-1,k_{m-1}}}}=\sum_{i=1}^{k_{m}}\frac{a_{m,i}}{2^{a_{m,i}}}.
$$
Thus, the number $x$ will have at least $m$ representations
$$
\sum_{i=1}^{k_{1}}\frac{a_{1,i}}{2^{a_{1,i}}},\quad \sum_{i=1}^{k_{1}-1}\frac{a_{1,i}}{2^{a_{1,i}}}+\sum_{i=1}^{k_{2}}\frac{a_{2,i}}{2^{a_{2,i}}},\ldots,\quad \sum_{j=1}^{m-1}\left(\sum_{i=1}^{k_{j}-1}\frac{a_{j,i}}{2^{a_{j,i}}}\right)+\sum_{i=1}^{k_{m}}\frac{a_{m,i}}{2^{a_{m,i}}}.
$$

We take $x=\frac{8}{2^8}=\frac{1}{32}$ and applying our greedy strategy we compute
\begin{align*}
\frac{1}{32}=&\frac{9}{2^{9}}+\frac{10}{2^{10}}+\frac{12}{2^{12}}+\frac{14}{2^{14}}+\frac{18}{2^{18}}+\frac{19}{2^{19}}+\\
             &\frac{21}{2^{21}}+\frac{22}{2^{22}}+ \frac{24}{2^{24}}+\frac{26}{2^{26}}+\frac{29}{2^{29}}+\frac{30}{2^{30}}+\frac{32}{2^{32}},
\end{align*}
i.e., $k_{1}=13, a_{i,k_{1}}=32$. Further values of $k_{i}$ and $a_{i,k_{i}}$ for $i\leq 9$ are as follows
$$
\begin{array}{c|ccccccccc}
  i           & 1 & 2   & 3 & 4 & 5 & 6 & 7 & 8 & 9 \\
  \hline
  k_{i}       & 13 & 9  & 169 & 5919  & 71826  & 252200 & 182973  & 10861   & 1195089 \\
  \hline
  a_{i,k_{i}} & 32 & 46 & 392 & 12230 & 155942 & 659488 & 1025582 & 1047128 & 3437088
\end{array}
$$
Due to size of the sets $\{a_{k_{j},i}:\;i=1,\ldots, k_{j}\}, j=2,\ldots,9$, we do not present them in full.

By adding the value of the series $\sum_{i=1}^{\infty}\frac{pi+q}{2^{pi+q}}=\frac{(q+p)2^p-q}{2^{q}(2^p-1)^2}$, where $p, q\in\N_{+}$ are chosen that $p+q>3437088$, to found representations, we get the statement of our theorem.
\end{proof}

With the data needed to get Theorem \ref{10000} we observed that the behaviour of $k=k(n)$ and $a_{k}=a_{k}(n)$ behaves quite irregular. For example, from our numerical data we collected the following peak (or jump) values of $k$.

\begin{equation*}
\begin{array}{|l|l|l|c|}
\hline
  n        & k(n)    & a_{k}(n) & \max\{k(i):\;i< n\} \\
\hline
  56       & 6092    & 12230   & 189 \\
  3113     & 13370   & 29752   & 6092 \\
  3817     & 76072   & 155942  & 13370 \\
  5588     & 460536  & 226913  & 76072 \\
  \hline
\end{array}
\end{equation*}
\begin{center} Table 2. Peak values among values of $k=k(n)$ \end{center}

In the picture below we also present the graph of the function
$$
k:\;\N_{\geq 2}\ni n\mapsto k(n)\in\N.
$$

\begin{figure}[h]\label{Pic1} 
       \centering
         \includegraphics[width=4in]{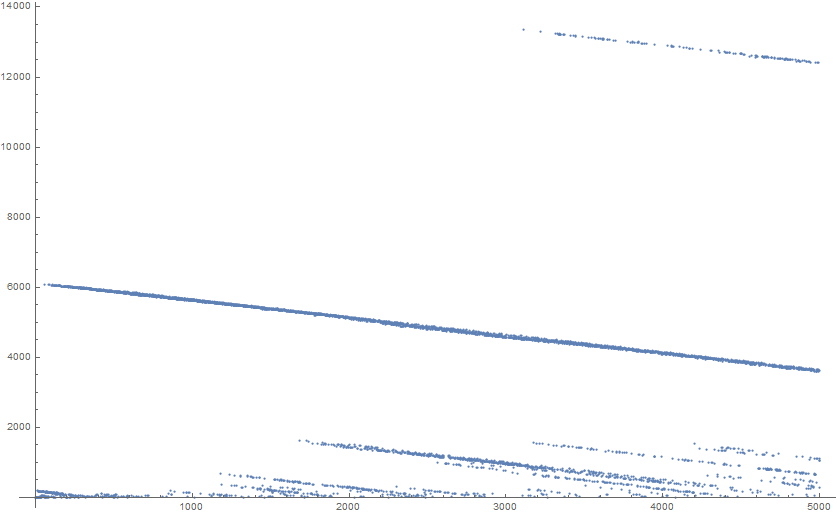}
        \caption{The value of $k$ such that $\frac{n}{2^{n}}=\sum_{i=1}^{k}\frac{a_{i}}{2^{a_{i}}}$ for some $a_{1},\ldots, a_{k}\in\N$ and $n\leq 5000$}
       \label{fig:disc1}
    \end{figure}

Based on our numerical data we formulate the following.

\begin{conj}\label{conj1}
If the Diophantine equation
$$
\frac{n}{2^{n}}=\sum_{i=1}^{k}\frac{a_{i}}{2^{a^{i}}}
$$
has a solution $(n, k, a_{1},\ldots, a_{k})$ with $a_{1}<a_{2}<\ldots <a_{k}$, then $k+n\leq a_{k}\leq 2(k+n)$. In particular $a_{k}\leq 4(2^{k}-1)$.
\end{conj}

On Figure 2 we present the behaviour of $a_{k}(n)/2(k+n)$, where the values of $a_{i}=a_{i}(n)$ come from our greedy algorithm.

\begin{figure}[h]\label{Pic2} 
       \centering
         \includegraphics[width=4in]{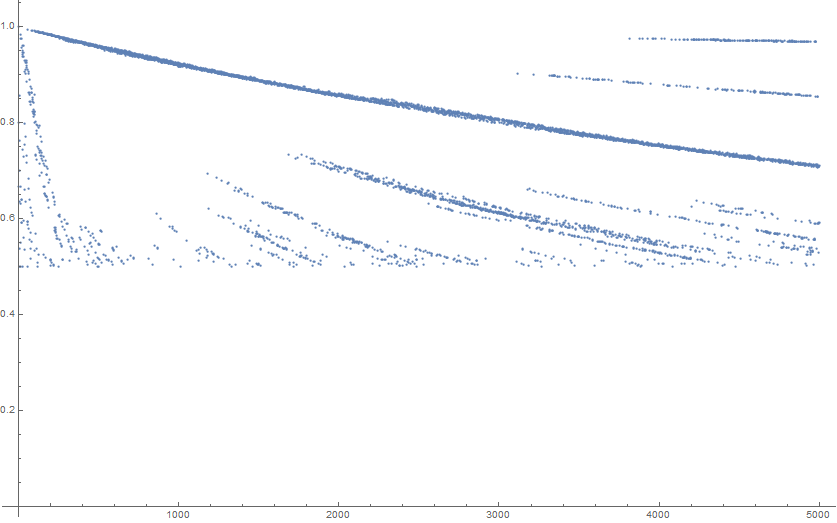}
        \caption{The value of the quotient $a_{k}(n)/2(k+n)$ from the greedy representation $\frac{n}{2^{n}}=\sum_{i=1}^{k}\frac{a_{i}}{2^{a^{i}}}$ for $n\leq 5000$}
       \label{fig:disc2}
    \end{figure}

On Figure 3 we also present the graph of the function $k(n)/n$.

\begin{figure}[h]\label{Pic4} 
       \centering
         \includegraphics[width=4in]{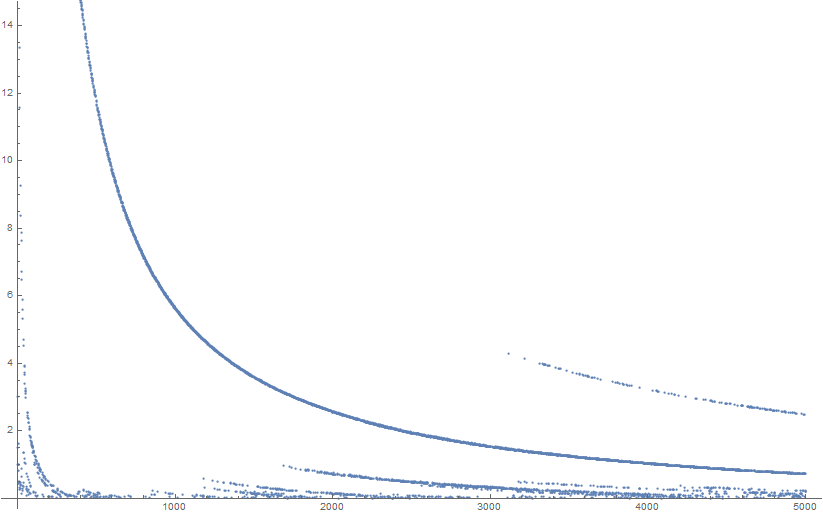}
        \caption{Plot of the ratio $k(n)/n$ coming from the greedy algorithm for $n\leq 5000$}
       \label{fig:disc4}
    \end{figure}

\begin{rem}
{\rm First of all let us observe that the value $n+k$ in the lower bound cannot be replaced by nothing greater. Indeed, if $n=2^{k+1}-k-2$ then we get the exact value $a_{k}=n+k$.

The upper bound for $a_{k}$ stated in the above conjecture is reasonable. More precisely, it is easy to see that our statement is true under additional assumption $n\geq 2^{k}-k$. Indeed, we have
$$
\frac{n}{2^{n}}=\sum_{i=1}^{k}\frac{a_{i}}{2^{a_{i}}}\leq \sum_{i=1}^{k-1}\frac{n+i}{2^{n+i}}+\frac{a_{k}}{2^{a_{k}}}.
$$
Equivalently, we have the inequality
$$
\frac{n+k+1-2^{k}}{2^{n+k-1}}\leq \frac{a_{k}}{2^{a_{k}}}.
$$
Let us assume that $a_{k}>2(k+n)$ and $2^{k}-k\leq n\leq 2^{k+1}-k-2$. Then, we have the inequality
$$
\frac{n+k+1-2^{k}}{2^{n+k-1}}\leq \frac{a_{k}}{2^{a_{k}}}\leq \frac{2(n+k)}{2^{2(n+k)}}\;\Longleftrightarrow \; 2^{n+k+1}(n+k+1-2^{k})\leq 2(n+k).
$$
Using the lower bound $2^{k}-k\leq n$ on the left hand side of the inequality and the upper bound $n\leq 2^{k+1}-k-2$ on the right hand side we get
$$
2^{n+k+1}\leq 2^{n+k+1}(n+k+1-2^{k})\leq 2(n+k)\leq 2(2^{k+1}-k-2+k)=4(2^{k}-1)
$$
and thus $2^{n+k-1}\leq 2^{k}-1$ - a contradiction.
}
\end{rem}


\noindent {\bf Acknowledgments.}
The research of the first author was partially supported in part by grants ANN130909, K115479 and of the Hungarian National Foundation for Scientific Research. The research of the second author was partially supported by the grant of the Polish National Science Centre no. UMO-2019/34/E/ST1/00094.

\bigskip

\noindent Szabolcs Tengely, Mathematical Institute, University of Debrecen, P.O.Box 12, 4010 Debrecen, Hungary\\
e-mail:\;{\tt tengely@science.unideb.hu}
\bigskip

\noindent  Maciej Ulas, Jakub Zygad{\l}o, Jagiellonian University, Faculty of Mathematics and Computer Science, Institute of Mathematics, {\L}ojasiewicza 6, 30 - 348 Krak\'{o}w, Poland\\
e-mail:\;{\tt $\{$maciej.ulas,jakub.zygadlo$\}$@uj.edu.pl}

 \end{document}